\documentclass[12pt]{amsart}
\usepackage{amsmath}
\usepackage{mathtools}

\usepackage{amsfonts}
\usepackage{amsthm}
\usepackage{autobreak}
\usepackage[english]{babel}
\usepackage{bold-extra}
\usepackage{hyperref}
\newtheorem{definition}{Definition}

\newtheorem{cor}[definition]{Corollary}

\newtheorem{lemma}[definition]{Lemma}
\newtheorem{theorem}[definition]{Theorem}
\newtheorem*{maintheorem}{Main theorem}
\newtheorem*{lemmaun}{Lemma}

\newtheorem*{notations}{Notations}

\numberwithin{definition}{section}
\allowdisplaybreaks

\def\B{{\mathbb{B}}}
\def\S{{\mathbb{S}}}
\def\D{{\mathbb{D}}}

\def\C{{\mathbb{C}}}

\def\N{{\mathbb{N}}}

\def\T{{\mathbb{T}}}
\def\UU{{\mathcal{U}}}
\def\TT{{\mathcal{T}}}

\def\RR{{\mathcal{R}}}
\def\PP{{\mathcal{P}}}

\def\BB{{\mathcal{B}}}
\def\QQ{{\mathcal{Q}}}

\newcommand{\vp}{\varphi}

\usepackage{multicol}
\usepackage[T1]{fontenc}
\usepackage[utf8]{inputenc}
\usepackage[a4paper]{geometry}
\usepackage{amssymb,amsmath}
\usepackage{version}
\usepackage{mathtools} 
\usepackage{enumerate}
\usepackage{dsfont}
\usepackage{mathdots}
\makeatletter

\usepackage{latexsym}
\usepackage{color}

\begin{document}
	
\title{Bloch functions with wild boundary behaviour in $\C^N$}

	\author{St\'{e}phane Charpentier, Nicolas Espoullier, Rachid Zarouf}
	
	\address{St\'ephane Charpentier, Aix-Marseille Univ, CNRS, I2M, Marseille, France}
	\email{stephane.charpentier.1@univ-amu.fr}
	
		\address{Nicolas Espoullier, Aix-Marseille Univ, CNRS, I2M, Marseille, France}
	\email{nicolas.espoullier@etu.univ-amu.fr}
	
	\address{Rachid Zarouf, Aix Marseille Univ, Lab ADEF, Campus Univ St Jerome,52 Ave Escadrille Normandie, F-13013 Marseille, France}
	\email{rachid.zarouf@univ-amu.fr}
	
	\address{}
	
	

	\begin{abstract}We prove the existence of functions $f$ in the Bloch space of the unit ball $\B_N$ of $\C^N$ with the property that, given any measurable function $\vp$ on the unit sphere $\S_N$, there exists a sequence $(r_n)_n$, $r_n\in (0,1)$, converging to $1$, such that for every $w\in \B_N$,
		\[
		f(r_n(\zeta -w)+w) \to \vp(\zeta)\text{ as }n\to \infty\text{, for almost every }\zeta \in \S_N. 
		\]
	The set of such functions is residual in the little Bloch space. A similar result is obtained for the Bloch space of the polydisc.
	\end{abstract}
	
	\maketitle
	
\section{Introduction and statements of the results}

Our purpose is to show the existence of functions in the Bloch space of the unit ball and of the polydisc of $\C^N$ that carry a quite erratic non-tangential boundary behaviour.

Throughout the paper, we will denote by $\D:=\left\{z\in \C:\, |z|<1\right\}$,
\[
\B_N:=\{z=(z_1,\ldots,z_N)\in \C^N:\, |z|^2:=|z_1|^2+\ldots + |z_N|^2<1\}
\]
and $\D^N:=\D\times \ldots \times \D$ (with $N\geq 1$) the open unit disc of the complex plane $\C$, the open unit ball and the polydisc of $\C^N$, respectively. The notations $\T$, $\S_N$ and $\T^N$ will stand for the unit circle, the unit sphere and the distinguished boundary of $\D^N$. Without possible confusions, we shall indifferently denote by $|\cdot|$ the modulus of a complex number and the euclidean norm in $\C^N$.

Given a domain $D\subset \C^N$, we denote by $H(D)$ the Fréchet space of all functions holomorphic in $\D$, endowed with the topology of locally uniform convergence. For $f\in H(\B_N)$, the radial derivative of $f$ is defined by
\[
\RR(f)=\sum_{k=1}^Nz_k\frac{\partial f}{\partial z_k}.
\]
The Bloch space $\BB(\B_N)$ of $\B_N$ consists in all functions $f$ holomorphic on $\B_N$ such that
\[
\sup_{z\in \B_N}(1-|z|^2)|\RR(f)(z)|<\infty.
\]
Endowed with the norm $\|f\|_{\BB(\B_N)}:= |f(0)| + \sup_{z\in \B_N}(1-|z|^2)|\RR(f)(z)|$, it is a Banach space. In particular, a function $f\in H(\D)$ belongs to the Bloch space $\BB(\D)$ of the unit disc provided that $\sup_{|z|<1}(1-|z|^2)|f'(z)|<\infty$. We then define the Bloch space $\BB(\D^N)$ of the polydisc as the subspace of $H(\D^N)$ consisting of those $f$ for which
\[
\sup_{z\in \D^N}\sum_{k=1}^N(1-|z_k|^2)\left|\frac{\partial f}{\partial z_k}(z)\right|<\infty.
\]
The quantity $\|f\|_{\BB(\D^N)}:= |f(0)| +\sup_{z\in \D^N}\sum_{k=1}^N(1-|z_k|^2)\left|\frac{\partial f}{\partial z_k}(z)\right|$ is a norm that turns $\BB(\D^N)$ into a Banach space.

The little Bloch spaces $\BB_0(\D)$, $\BB_0(\B_N)$ and $\BB_0(\D^N)$ of the unit disc, the unit ball and the polydisc, respectively, are defined as the closure of all polynomials in the corresponding Bloch space. For any $N\geq 1$, $\BB_0(\B_N)$ coincides with the space of all functions $f$ in $H(\B_N)$ such that
\[
\lim_{|z|\to 1^-}(1-|z|^2)|\RR(f)(z)|=0.
\]
It is known that a similar assertion fails for $\BB_0(\D^N)$, $N\geq 2$ \cite{Timoney1980II}. When the context will be clear, we will refer to a function in one of the Bloch spaces introduced above as a \textit{Bloch function}. For equivalent definitions of these spaces and their classical properties, we refer the reader to \cite[Chapter 3]{Zhubook2005} and \cite{Timoney1980I,Timoney1980II}.

The search for Bloch functions with \emph{irregular} radial boundary behaviour is a rather classical topic. In the unit disc, it is rather easy to build a function in $\BB(\D)$ with finite radial limit at no point of $\T$. Indeed, the lacunary series $\sum_k z^{2^k}$ is a Bloch function that has finite radial limits at no points of $\T$, by a tauberian theorem due to Hardy and Littlewood \cite{HardyLittlewood1926}. In passing, by Plessner's theorem, this implies that the image by such functions of almost all radii in $\D$ is dense in $\C$. In the unit ball of $\C^N$, proving the existence of Bloch functions with radial limits at no point of $\S_N$ is far more involved. Ryll and Wojtaszczyk introduced the so-called Ryll-Wojtaszczyk  homogeneous polynomials in order to build a function in $\BB(\B_N)$ with infinite radial limits at \emph{almost} every point of $\S_N$ \cite{RyllWojtaszczyk1983}. Later, using modifications of those polynomials by Aleksandrov \cite{Aleksandrov1986}, Ullrich built a function $f\in \BB(\B_N)$ such that for \emph{every} $\zeta \in S_N$, $|f(r\zeta)|\to \infty$, as $r\to 1$ \cite{Ullrich1988}. Note that, by \cite[Theorem 3]{BagemihlSeidel1961}, for any function $f$ in $\BB(\B_N)$ with finite radial limit at no point of $\S_N$, there exists a dense set of points $\zeta$ in $\S_N$ such that $|f(r\zeta)| \to \infty$ as $r\to 1$.

\medskip

The main contribution of this note is to prove that \emph{quasi-all} functions in $\BB_0(\B_N)$ or $\BB_0(\D^N)$ (where quasi-all should be understood in the sense of Baire category theorem) carry an even more erratic non-tangential boundary behaviour, near \emph{almost every} point of $\S_N$ or $\T^N$. We recall that a subset of a separable Fréchet space is said residual if it contains a dense countable intersection of open sets. Let us denote by $m_N$ the normalised Lebesgue measure on $\S_N$ and by $\sigma_N:=m_1\times\ldots\times m_1$ the normalised Lebesgue measure on $\T^N$.

\begin{maintheorem}Let $(r_n)_n$ be a sequence of real numbers in $(0,1)$, converging to $1$ as $n\to \infty$.
	\begin{enumerate}
		\item There exists a residual subset of $\BB_0(\B_N)$ consisting of functions $f$ that satisfy the following property:
		\[
		(\mathcal{Q}_{\B_N}):\,\begin{array}{l}\text{Given any }m_N\text{-measurable function }\vp\text{ on }\S_N\text{, there exists a sequence }(n_k)_k,\,n_k\in\N,\\
			\text{such that for any }w\in \B_N\text{ and }m_N\text{-a.e. }\zeta \in \S_N\text{, }\\
			\hfill f(r_{n_k}(\zeta-w)+w)\to \vp(\zeta) \text{ as }k\to \infty.\hfill
			
		\end{array}
		\]
		\item There exists a residual subset of $\BB_0(\D_N)$ consisting of functions $f$ that satisfy the following property:
		\[
		(\mathcal{Q}_{\D^N}):\,\begin{array}{l}\text{Given any }\sigma_N\text{-measurable function }\vp\text{ on }\T_N\text{, there exists a sequence }(n_k)_k,\,n_k\in\N,\\
			\text{such that for any }w\in \D_N\text{ and }\sigma_N\text{-a.e. }\zeta \in \T_N\text{, }\\
			\hfill f(r_{n_k}(\zeta-w)+w)\to \vp(\zeta) \text{ as }k\to \infty.\hfill
			
		\end{array}
		\]
	\end{enumerate}
\end{maintheorem}

Similar results were already stated in other spaces of holomorphic functions. For example, Bayart proved (1) where the Bloch space $\BB_0(\B_N)$ is replaced with the Fréchet space $H(\B_N)$, or by any \emph{little growth space} $H_{w,0}^{\infty}$ of the ball  \cite{Bayart2005}. Given a (continuous) weight function $w(t):[0,1)\to (0,+\infty)$ with $w(0^+)=0$, let us recall that the growth space $H_{w}^{\infty}$ consists of all functions $f$ holomorphic in $\B_N$ such that
\[
\sup_{z\in \B_N}w(1-|z|)|f(z)|< \infty.
\]
The space $H_{w,0}^{\infty}$ can be defined as the closure of set of polynomials in $H_{w}^{\infty}$, and it can be checked that it coincides with the set of those $f$ for which $w(1-|z|)|f(z)|\to 0$ as $|z|\to 1$. Since there exist weights $w$ for which $H_{w}^{\infty}$ is contained in all Bergman spaces $A^p$, $p\geq 1$, the space $\cap_{p\geq 1}A^p$ contains a (lot of) function(s) satisfying the property $(\mathcal{Q}_{\B_N})$. In \cite{CharpentierKosinski2021}, an extension of this result to $H(D)$, where $D$ is a pseudoconvex domain of $\C^N$, is obtained. We observe that in the space $H(\D)$, it is known \cite{Charpentier2020} that quasi-all functions satisfy a much stronger property. Finally, we mention that our theorem answers Question 5.7 in \cite{CharpManoMaroI}.

\medskip

Let us now comment on the proof. The above statement typically falls within the theory of universality \cite{BayartGrosseErdmanNestoridisPapadimitropoulos2008,GrosseErdmann1999}, and it is now well-understood that most of results of this kind follow from a Baire category argument  and eventually reduce to a suitable \textit{simultaneous approximation}. The general idea is that, if we are given a (reasonable) separable Fréchet space $X$ continuously embedded in $H(\B_N)$ (the strategy is the same for the polydisc), in order to build a function $f$ in $X$ with some prescribed boundary behaviour $m_N$-a.e. on $\S_N$, it is enough, roughly speaking, to find a function $f$ in $X$ that simultaneously approximates $0$ in $X$ and any given continuous function on some \textit{large} subset of $\S_N$. We point out that simultaneous polynomial approximation in spaces of holomorphic functions is an independent active topic of research. We refer to the seminal thesis work of Khrushchev \cite{Khrushchev1978} and to the recent paper \cite{Limani2024} (see also the references therein).

\medskip

Following this strategy, in order to prove the assertion (1) in our main theorem, we will prove a slightly more precise version of the following simultaneous approximation lemma (see Lemma \ref{lemma-simul-approx-Bloch} ).

\begin{lemmaun}[Simultaneous approximation in the Bloch space of the ball]For any $\varepsilon>0$, any $g\in \BB_0(\B_N)$, and any function $\vp$ continuous on $\S_N$, there exist a measurable set $E\subset \S_N$, with $m_N(E)\geq 1-\varepsilon$, and a polynomial $f$ such that
	\[
	\|f-g\|_{\BB(\B_N)}\leq \varepsilon \quad \text{and} \quad \sup_{z\in E}|f(z)-\vp(z)|\leq \varepsilon.
	\]
\end{lemmaun}

A similar statement will be proved for $\D^N$ (see Lemma \ref{lemma-simul-approx-Bloch-polydisc}). For any holomorphic function space $X$ on $\B_N$, that contains polynomials, if the set of all polynomials is dense in $X$ and contained in the pointwise multipliers of $X$, then proving the above lemma reduces to the case where $g=0$ and $\vp =1$. Now, a general approach to simultaneously approximate $0$ in $X$ and $1$ on a \textit{large} subset of $\S_N$ consists in selecting the approximant $f$ as equal to $P\circ I$, where $P$ is a polynomial (usually given by Oka-Weyl theorem - or Runge's theorem in $\C$) that simultaneously approximate uniformly $1$ on arbitrarily \textit{large} proper compact subset of $\S_N$ and $0$ in a compact subset of $\B_N$, and where $I$ is an inner mapping chosen so that $P\circ I$ is small in $X$ norm, and so that it \textit{transports} the approximation property of $P$ on (another) \textit{large} subset of $\S_N$. For example, in \cite{Bayart2005}, the author made use of a similar simultaneous approximation result for $H_{w,0}^{\infty}$, that was due to Iordan \cite{Iordan1989}. Therein, the polynomial $P$ is chosen using a Mergelyan's type result for the ball \cite{HakimSibony1987}, and the inner mapping is $z\Phi^k(z)$, $z\in\B_N$, where $\phi$ is any inner function vanishing at $0$, and the power $k$ large enough (depending on the weight $w$).

This approach is followed by Limani \cite[Section 4.2]{Limani2024}, who basically proves the above lemma for $N=1$. The main difficulty consists in showing the existence of an inner function $I$ such that $\|P\circ I\|_{\BB}$ is arbitrarily small, which essentially reduces to the existence of an inner function $I$ with arbitrarily small Bloch norm. It turns out that this can be achieved using a deep result of Aleksandrov, Anderson and Nicolau \cite{AAN1999}. We mention that this allows Limani to prove the existence of so-called Menshov universal functions in the Bloch space of $\D$, answering a question posed in \cite{BeiseMuller2016}. We also refer to \cite{Khrushchev2020,Khrushchev2023} for results of a similar nature in other Banach spaces of holomorphic functions on $\D$.

In order to prove our simultaneous approximation lemma for the Bloch space of the ball, a first difficulty is that one does not have a proper Mergelyan's theorem, and the second one is that inner functions are usually much more difficult to construct. However, a combination of simple geometric ideas, that served already in \cite{CharpentierKosinski2021}, and a partial generalisation of Aleksandrov-Anderson-Nicolau's result, due to Doubtsov \cite{Doubtsov2000}, will allow us to obtain the main theorem.

Let us say a few words on the case of the polydisc, which slightly differs from the case of the ball ($N\geq 2$). Indeed, the only mutlipliers of $\BB_0(\D^N)$ are the constant functions \cite[Corollary 3.6]{AllenColonna2009} and thus we cannot directly follow the strategy described for $\B_N$. However, this can be overcome working in the unit disc and building our approximating polynomials as a product of polynomials of one variable.

\medskip

The paper is organised as follows: the second section is dedicated to the results of simultaneous approximation in the Bloch space of the ball and the polydisc. The proof of our main theorem is presented in Section 3.  A final section is devoted to commenting briefly on possible extensions and variations.

\begin{notations}{\rm The notation $A\lesssim_P B$ shall be used when there exists a positive constant $C$, depending on the parameter $P$, such that $A\leq CB$. If the constant $C$ does not depend on any parameter relevant in the context, we shall simply use the symbol $\lesssim$.}
\end{notations}

\section{Simultaneous approximation in the Bloch space}

For any compact set $E$ in $\C^N$, we will denote by $C(E)$ the Banach space of all continuous functions on $E$, endowed with the supremum norm, denoted by $\|\cdot\|_{\infty,E}$.

\subsection{Simultaneous approximation in the Bloch space of the ball}\label{Section-lem-ball}


In this part, our aim is to prove the following approximation lemma.

\begin{lemma}\label{lemma-simul-approx-Bloch}Let $\varepsilon>0$, let $L$ be a compact subset of $\B_N$, let $g\in \BB_0(\B_N)$, and let $\vp \in C(\S_N)$. There exist a set $E\subset \S_N$ with $m_N(E)\geq 1-\varepsilon$ and a polynomial $f\in \BB_0(\B_N)$, such that
	\begin{enumerate}[(i)]
		\item $\|f-g\|_{\BB}< \varepsilon$;
		\item $\|f-\vp\|_{\infty,E}< \varepsilon$.
	\end{enumerate}
\end{lemma}

Since the polynomials are dense in $\BB_0(\B_N)$, it is enough to prove the lemma for $g=0$, upon replacing $\vp$ by $\vp - Q$, with $Q$ close to $g$ in $\BB_0(\B_N)$. The proof is based on two results. The first one is geometric and follows immediately from \cite[Lemma 3.6]{CharpentierKosinski2021}, for e.g.

\begin{lemma}\label{lem-geo}Let $\varepsilon>0$ and let $U$ be a domain in $\C^N$ such that $U\cap \S_N \neq \emptyset$. There exists a measurable subset $E$ of $U\cap\S_N$ such that:
	\begin{enumerate}[(i)]
		\item $m_N(U\cap \S_N \setminus E)< \varepsilon$;
		\item for any compact set $L\subset \B_N$, the set $E \cup L$ is polynomially convex.
	\end{enumerate}
\end{lemma}

The second one deals with the existence of inner functions in the ball with a specific behaviour at the boundary. It is a combination of two results by Aleksandrov-Anderson-Nicolau \cite{AAN1999} and Doubtsov \cite{Doubtsov2000}.

\begin{theorem}\label{thm-inner-function-ball}Let $\eta >0$ be fixed. There exists a non-constant inner function $I:\B_N\to \C$ such that
	\[
	\frac{(1-|z|^2)|\RR I(z)|}{1-|I(z)|^2}< \eta,\quad z\in \B_N.
	\]
In particular, $\|I\|_{\BB}< \eta$.
\end{theorem}

\begin{proof}Let us fix $\eta >0$. Let $\mu$ be a pluriharmonic probability measure such that the slices $\mu_{\zeta}$ are uniformly symmetric and singular for all $\zeta \in \S_N$. Such measures exist, by Corollary 2.3 in \cite{Doubtsov2000}. Then define the non-constant singular inner function
	\[
	\Theta(z)=\exp \left(-\int_{\S_N}\left(\frac{2}{(1-\left<z,\zeta\right>)^n}-1\right)d\mu(\zeta)\right).
	\]
Then, again by \cite{Doubtsov2000}, one has
\[
\lim_{|z|\to 1}	\frac{(1-|z|^2)|\RR \Theta(z)|}{1-|\Theta(z)|^2}=0.
\]
In particular, there exists a positive constant $M<\infty$ such that
\[
\frac{(1-|z|^2)|\RR \Theta(z)|}{1-|\Theta(z)|^2}\leq M,\quad z\in \B_N.
\]
By Theorem 1 in \cite{AAN1999} applied with $\phi(t)=\eta t/(2M)$, there exists an inner function $g:\D\to \D$ such that
\[
\frac{(1-|w|^2)|g'(w)|}{1-|g(w)|^2} \leq \frac{\eta}{2M}, \quad w\in \D.
\]
Let us set $I=g\circ \Theta$. Then $I$ is a non-constant inner function on $\B_N$ such that, for any $z\in \B_N$,
\[
\frac{(1-|z|^2)|\RR I(z)|}{1-|I(z)|^2} = \frac{(1-|z|^2)|g'(\Theta(z))||\RR \Theta(z)|}{1-|g(\Theta(z))|^2}\leq \frac{\eta}{2M} \frac{(1-|z|^2)|\RR \Theta(z)|}{1-|\Theta(z)|^2} < \eta.
\]
\end{proof}

We are now ready to prove Lemma \ref{lemma-simul-approx-Bloch}.

\begin{proof}[Proof of Lemma \ref{lemma-simul-approx-Bloch} with $g=0$] Let $\varepsilon>0$, let $L$ be a compact subset of $\B_N$, and let $\vp \in C(\S_N)$.
	
First, by uniform continuity of $\vp$, there exist open domains $U_1,\ldots,U_l$ in $\C^N$, pairwise disjoint, with $U_j\cap \S_N\neq \emptyset$ for any $j$, and complex numbers $c_1,\ldots, c_l$, such that
\begin{enumerate}[(a)]
	\item $m_N(\S_N\setminus (\overline{U_1}\cup\ldots\cup \overline{U_l}))\leq \varepsilon/5$;
	\item $|\vp(z)-c_k|\leq \varepsilon/2$ for any $z\in \overline{U_k}\cap \S_N$ and any $k\in \{1,\ldots,l\}$.
\end{enumerate}

Now, applying Lemma \ref{lem-geo} to each $U_k$, we choose a compact set $F\subset \S_N\cap(U_1\cup\ldots\cup U_l)$ with  $\sigma(F)\geq 1-\varepsilon/4$, such that $F\ \cup \{0\}$ is polynomially convex. Since the continuous function equal to $0$ at $0$ and to $c_k$ on $U_k$, $k\in \{1,\ldots,l\}$, extends holomorphically to an open neighborhood of $\{0\}\cup \overline{U_1}\cup\ldots\cup \overline{U_l}$, the Oka-Weyl theorem (see \cite{Stoutbook2007}, e.g.) ensures the existence of a polynomial $Q$ such that
\[
|Q(z)-\vp(z)| < \varepsilon/2,\quad z\in F.
\]

Let also $P$ be a polynomial such that
\begin{equation}\label{eq-lem-P}
|P(0)|< \frac{\varepsilon}{2\|Q\|_{\infty,\B_N}}\quad \text{and}\quad |P(z)-1|< \frac{\varepsilon}{2\|Q\|_{\infty,\B_N}}, z\in F.
\end{equation}
	
Let now $I$ be given by Theorem \ref{thm-inner-function-ball} for some $\eta>0$ whose value will be fixed later. We define the mapping $J:z=(z_1,\ldots,z_n) \mapsto I(z)z= (z_1I(z),\ldots,z_nI(z))$, $z\in \B_N$. It is clear that $J$ is an inner mapping from $\B_N$ to $\B_N$ that vanishes at $0$. Now, a simple calculation shows that
	\[
	\RR(P\circ J)(z)= \left(\RR(I)(z) + I(z)\right)\sum_{i=1}^nz_i\frac{\partial P}{\partial z_i}(J(z)),
	\]
Hence, since $P$ is a polynomial, we get for any $z\in \B_N$,
	\begin{eqnarray*}
	|\RR(P\circ J)(z)| & \leq & \sum_{i=1}^n\left|\frac{\partial P}{\partial z_i}(J(z))\right|\left(|I(z)|+|\RR(I)(z)|\right)\\
	& \lesssim_P & |I(z)|+|\RR(I)(z)|.
	\end{eqnarray*}
Moreover, as it is well-known,
	\[
	|I(z)|\lesssim \|I\|_{\BB(\B_N)}\log(\frac{1}{1-|z|^2}),\quad z \in \B_N.
	\]
	Then for any $z\in \B_N$,
	\[
	(1-|z|^2)|\RR(P\circ J)(z)|\lesssim_P (1-|z|^2)|\RR(I)(z)| + (1-|z|^2)\log(\frac{1}{1-|z|^2})\|I\|_{\BB(\B_N)}.
	\]
	By Theorem \ref{thm-inner-function-ball}, the function $P\circ J$ belongs to $\BB_0(\B_N)$, and one has $\|I\|_{\BB(\B_N)}< \eta$. Therefore
	\[
	(1-|z|^2)|\RR(P\circ J)(z)|\lesssim_P \eta\left(1 + (1-|z|^2)\log(\frac{1}{1-|z|^2})\right),\,z\in \B_N,
	\]
	hence $(1-|z|^2)|\RR(P\circ J)(z)|\lesssim_P \eta$, $z\in \B_N$.
	
	\medskip
	
	We set $f=Q.(P\circ J)$. Using the fact that the multiplication by $Q$ is a bounded linear operator on $\BB_0(\B_N)$ \cite[Theorem 3.21]{Zhubook2005}, and upon choosing $\eta$ small enough, we deduce by the previous that
	\[
	(1-|z|^2)|\RR(f)(z)|< \varepsilon/2, \quad z\in \B_N.
	\]
	Moreover, since $J$ vanishes at $0$, by \eqref{eq-lem-P}, we have $|f(0)|< \varepsilon/2$. All in all we have proven the inequality $\|f\|_{\BB(\B_N)}< \varepsilon$.

\medskip

	To finish the proof, observe that, by \eqref{eq-lem-P} and the definition of $Q$, for any $z\in J^{-1}(F)\cap F$, we have
	\[
	|f(z)-\vp(z)|\leq |Q(z)||(P(J(z))-1)| + |Q(z)-\vp(z)|< \varepsilon.
	\]
	Since $J = Iz$, with $I$ inner, Loewner's lemma (see \cite[Example 14.2]{Rudinbook1986}) implies
	\[
	m_N(J^{-1}(F))=\sigma (F)\geq 1-\varepsilon/4,
	\]
	hence $m_N(J^{-1}(F)\cap F)\geq 1-\varepsilon/2$. We set $E=J^{-1}(F)\cap F$.
	
	To finish, it remains to explain why $f$ can be chosen as a polynomial. To do so, let us recall that $H^{\infty}$ is continuously embedded into $\BB_0$ and that $\|f_r-f\|_{\BB}\to 0$ as $r\to 1$ for any $f\in \BB_0$,  where $f_r$ denote the dilate of $f$ defined by $f_r(z)=f(rz)$, $z\in \B_N$ \cite{Zhubook2005}.  Therefore, we may choose $r\in (0,1)$ close enough to $1$ and approximate $f_r$ by a polynomial uniformly on $\overline{\B_N}$, to get our conclusion.
\end{proof}


For $\eta >0$, we will denote by $I_{\eta}:\B_N \to \B_N$ any inner mapping given by Theorem \ref{thm-inner-function-ball}, and let $J_{\eta}(z)= I(z)z$, $z\in \B_N$. The following corollary is a straightforward consequence of the proof of Lemma \ref{lemma-simul-approx-Bloch}.

\begin{cor}\label{cor-approx-disc}Let  $d\geq 1$, $\varepsilon_1 >0$, $\varepsilon_2>0$ and let $\vp \in C(\S_n)$. There exist two polynomials $P$ and $Q$ such that, for any $\eta>0$, there exists a compact set $E\subset \S_N$, with $m_N(E)\geq 1-\tilde{\varepsilon_2}$, such that, if we set $f:=Q(P\circ J_{\eta})$, then the following assertions hold:
	\begin{enumerate}
		\item $|f(0)| < \varepsilon_1$;
		\item $\|f\|_{\BB(\B_N)} < \eta + |f(0)|$;
		\item $\|f-\vp\|_{E,\infty}< \varepsilon_1$.
	\end{enumerate}
\end{cor}

In the next paragraph, we shall use this corollary to obtain a simultaneous approximation lemma for the polydisc.

\subsection{Simultaneous approximation in the Bloch space of the polydisc}\label{Section-lem-polydisc}


We shall prove the following Bloch simultaneous approximation lemma for the polydisc.

\begin{lemma}\label{lemma-simul-approx-Bloch-polydisc}Let $\varepsilon>0$, let $g\in \BB_0(\D^N)$ and let $\vp \in C(\T^N)$. There exist a set $E\subset \T^N$ with $\sigma_N(E)\geq 1-\varepsilon$ and a polynomial $f$, such that
	\begin{enumerate}
		\item $\|f-g\|_{\BB(\D^N)}\leq \varepsilon$;
		\item $\|f-\vp\|_{E,\infty}\leq \varepsilon$.
	\end{enumerate}
\end{lemma}

\begin{proof}Since the polynomials are dense in $\BB_0(\D^N)$, we may assume that $g$ is identically equal to $0$. Let $\varepsilon$, $L$, and $\vp$ be fixed as in the statement. By the Stone-Weierstrass Theorem, there exist $M>0$ and  functions $\vp_{j,l}\in C(\T)$, $j=1\ldots, N$ and $1\leq l\leq M$, such that
	\[
	\|\vp - (\vp_{1,1}\ldots\vp_{N,1}+\ldots+\vp_{1,M}\ldots\vp_{N,M})\|_{\infty,\T^N}<\varepsilon.
	\]
We shall use Corollary \ref{cor-approx-disc} to approximate each product $\vp_{1,l}\ldots\vp_{N,l}$, $1\leq l \leq M$, by some suitable polynomial. Let us fix $l=1$ and set
\[
C:=\prod_{j=1}^N(\|\vp_{j,1}\|_{\infty,\T}+1),
\]
Let also $\tilde{\varepsilon}>0$ whose value will be fixed later. We apply Corollary \ref{cor-approx-disc} to each $\vp_{j,1}$, $1\leq j \leq N$, with $d=1$, $\varepsilon_1=\min(1,\frac{\varepsilon}{2M},\frac{\varepsilon}{CNM})$ and $\varepsilon_2=\tilde{\varepsilon}$. This gives us $N$ pairs of polynomials $(P_{j,1},Q_{j,1})$, $1\leq j\leq N$, such that if we fix an $\eta >0$ satisfying
\begin{equation}\label{eq-fix-eta}
\eta <  \frac{\varepsilon}{2NM\max_{1\leq k\leq N}\sup_{(z_1,\ldots,z_N)\in \D^N}\prod_{j\neq k}|f_{j,1}(z_j)|},
\end{equation}
there exist compact sets $E_{j,1}$, $1\leq j\leq N$, with $\sigma_1(E_{j,1})\geq 1-\tilde{\varepsilon}$, such that setting $f_{j,1}:=Q_{j,1}(P_{j,1}\circ J_{\eta})$, we have:
\begin{enumerate}[(i)]
	\item $|f_{j,1}(0)| < \frac{\varepsilon}{2M}$;
	\item $\|f_{j,1}\|_{\BB(\B_N)} < \eta + |f_{j,1}(0)|$;
	\item $\|f_{j,1}-\vp_{j,1}\|_{\infty,E_{j,1}}< \frac{\varepsilon}{CNM}$.
\end{enumerate}
Furthermore, as explained at the very end of the proof of Lemma \ref{lemma-simul-approx-Bloch}, we may and shall assume that each function $f_{j,1}$, $1\leq j\leq N$, is a polynomial.

Now, we define the polynomial  $f_1$ of $N$ complex variables by
\[
f_1(z_1,\ldots,z_N):=f_{1,1}(z_1)\ldots f_{N,1}(z_N),
\]
and the compact subset $E_1:=E_{1,1}\times\ldots \times E_{N,1}$ of $\T^N$. We shall check the following two assertions:
\begin{enumerate}[(a)]
	\item $\|f_1\|_{\BB(\D^N)} < \varepsilon/M$;
	\item $\|f_1-\vp_{1,1}\ldots\vp_{N,1}\|_{\infty,E(1)} < \varepsilon/M$.
\end{enumerate}
First, notice that $|f_1(0)|<\varepsilon/(2M)$, by (i). Moreover, by (ii) and the choice of $\eta$, for any $(z_1,\ldots,z_N)\in \D^N$,
\begin{multline*}
\sup_{z\in \D^N}\sum_{k=1}^N(1-|z_k|^2)\left|\frac{\partial f_1}{\partial z_k}(z)\right|=\sup_{z\in \D^N}\sum_{k=1}^N(1-|z_k|^2)\left|f_{k,1}'(z_k)\right|\prod_{j\neq k}|f_{j,1}(z_j)| \\
\leq  \eta N\max_{1\leq k\leq N}\sup_{(z_1,\ldots,z_N)\in \D^N}\prod_{j\neq k}|f_{j,1}(z_j)|  < \frac{\varepsilon}{2M}
\end{multline*}
which gives (a).

Next, by (iii), for any $z=(z_1,\ldots,z_N)\in E_1$, we have
\begin{multline*}
|\vp_{1,1}\ldots\vp_{N,1}(z)-f_1(z)|\\
\leq \sum_{j=1}^N|f_{1,1}(z_1)\ldots f_{j-1}(z_{j-1})(f_{j,1}(z_j)-\vp_{j,1}(z_j))\vp_{j+1,1}(z_{j+1})\ldots \vp_{N,1}(z_{N})| \\
\leq \frac{\varepsilon}{CNM}\sum_{j=1}^N(\|\vp_{1,1}\|_{\infty,\T}+\varepsilon)\ldots(\|\vp_{j-1,1}\|_{\infty,\T}+\varepsilon)\|\vp_{1,1}\|_{\infty,\T}\ldots\|\vp_{j-1,1}\|_{\infty,\T}\\
\leq \frac{\varepsilon }{CM}\prod_{j=1}^N(\|\vp_{j,1}\|_{\infty,\T}+1)= \frac{\varepsilon}{M},
\end{multline*}
whence assertion (b).

Further, we repeat the previous argument for the other products  $\vp_{1,l}\ldots\vp_{N,l}$, $2\leq l \leq M$. At the end, this provides us with $M$ polynomials $f_1,\ldots,f_M$ and $M$ subsets $E_1,\ldots,E_M$ such that, if we set $E:=E_1\cap\ldots\cap E_M$ and $f:=f_1 +\ldots+f_M$, then
\begin{enumerate}
\item $\|f\|_{\BB(\D^N)} < \varepsilon$;
\item $\|f-(\vp_{1,1}\ldots\vp_{N,1}+\ldots+\vp_{1,M}\ldots\vp_{N,M}))\|_{\infty,E} < \varepsilon$.
\end{enumerate}
Moreover, upon choosing $\tilde{\varepsilon}$ small enough, we may assume $\sigma_N(E)\geq 1-\varepsilon$. The lemma is proven, since $\|\vp - (\vp_{1,1}\ldots\vp_{N,1}+\ldots+\vp_{1,M}\ldots\vp_{N,M})\|_{\infty,\T^N}<\varepsilon$.
\end{proof}

We note that Lemma \ref{lemma-simul-approx-Bloch} and Lemma \ref{lemma-simul-approx-Bloch-polydisc} coincide when $N=1$.

\section{Wild Bloch functions in the ball and the polydisc}

As it has been said in the introduction, showing the existence of functions with universal boundary behaviour reduces to proving simultaneous approximation. It turns out that this fact is a quite general one in the theory of universality \cite{BayartGrosseErdmanNestoridisPapadimitropoulos2008}. For the seek of brevity, we shall present an abstract result that, combined with Lemma \ref{lemma-simul-approx-Bloch}, will immediately imply our main results.

Let $(X,d_X)$ be a Fréchet space and let $Y$ be a separable topological vector space, metrizable with a translation-invariant metric $d_Y$. Let also $T_n:X\to Y$, $n\in \N$, be continuous linear maps. We say that the sequence $\TT:=(T_n)_n$ is universal if there exists $x\in X$ provided the set $\left\{T_n (x):\,n\in \N\right\}$ is dense in $Y$. More generally, we will say that a family  $\TT_{i\in I}$ of sequences $(T_n^i)_n$, $i\in I$, of continuous linear maps from $X$ to $Y$, is universal \emph{uniformly with respect to $i\in I$}, if there exists $x\in X$ such that, for any $\varepsilon>0$ and any $y\in Y$, there exists $n\in \N$ such that
\[
\sup_{i\in I}d(T_n^i(x),y)<\varepsilon.
\]
We denote by $\UU(\TT)$ the set of all universal vectors for $\TT$ and by $\UU(\TT_{i\in I})$ the set of all vectors that are universal for $\TT_{i\in I}$, uniformly with respect to $i\in I$. The following theorem is a very small improvement of a particular case of \cite[Theorem 27]{BayartGrosseErdmanNestoridisPapadimitropoulos2008}, 

\begin{theorem}\label{thm-bgnp-abstract}Let $\TT_{i\in I}$, be a family of sequences $(T_n^i)_n$, $i\in I$, of continuous linear maps from $X$ to $Y$. Assume that there exists a dense subset $G$ of $X$ such that, for any $x\in G$, there exists $y\in Y$ such that $(T_n^i(x))_n$ converges to $y$, uniformly with respect to $i\in I$. If the following two conditions hold:
	\begin{enumerate}
		\item \label{item1}For every $\varepsilon>0$ and every $y\in Y$, the set $\{x\in X:\,\sup_{i\in I}d_Y(T_n^i(x),y)< \varepsilon\}$ is open in $X$;
		\item \label{item3}For every $\varepsilon>0$ and every $y\in Y$, there exists $n\in \N$ and $x\in X$, such that
		\[
		d_X(x,0)<\varepsilon \quad\text{and}\quad \sup_{i\in I}d_Y(T_n^i(x),y)< \varepsilon,
		\]
	\end{enumerate}
then $\UU(\TT_{i\in I})$ is residual in $X$. 
\end{theorem}

\begin{proof}Even if it is very similar to some part of that of \cite[Theorem 27]{BayartGrosseErdmanNestoridisPapadimitropoulos2008}, we include it here. We first prove that (2) implies the following apparently stronger property:
	\[
	(2)':\quad\begin{array}{l}\text{For every } \varepsilon>0,\,y\in Y,\, w\in X,\text{ there exist } n\in \N\text{ and }x\in X\text{ such that}\\ d_X(x,w)<\varepsilon \quad\text{and}\quad \sup_{i\in I}d_Y(T_n^i(x),y)< \varepsilon.\\ 
	\end{array}
	\]
Let $\varepsilon>0,\,y\in Y,\, w\in X$ be fixed. By assumption, there exists $w_0\in G$, $v\in Y$ and $N\in \N$, such that $d_X(w,w_0)<\varepsilon/2$ and $\sup_{i\in I}d_Y(T_n^i(w_0),v)< \varepsilon/2$, for any $n\geq N$. By (2), there exists $n\in \N$ and $x_0\in X$ such that $d_X(x_0,0)<\varepsilon/2$ and $\sup_{i\in I}d_Y(T_n^i(x_0),y-v)<\varepsilon/2$, where $n$ can be chosen with $n\geq N$. Then, setting $x=x_0+w_0 \in X$, since the metrics are invariant by translation, we get $d_X(x,w)<\varepsilon$ and $\sup_{i\in I}d_Y(T_n^i(x),w)<\varepsilon$.

To conclude, let us fix a dense sequence $(y_l)_l$ in $Y$. We observe that
\[
\UU(\TT_{i\in I})\supset \bigcap_{k,l\in \N}\bigcup_{n\in \N}\{x\in X:\,\sup_{i\in I}d_Y(T_n^i(x),y_l)<\frac{1}{k}\}.
\]
Now the proof follows by the Baire Category theorem. Indeed, every set $\bigcup_{n\in \N}\{x\in X:\,\sup_{i\in I}d_Y(T_n^i(x),y_l)<\frac{1}{k}\}$ is open and dense in $X$, by (1) and (2)$'$, respectively. 
\end{proof}

\subsection{Wild Bloch functions in the ball}\label{sub-wild-ball}

We shall see that Lemma \ref{lemma-simul-approx-Bloch} implies that \eqref{item3} in Theorem \ref{thm-bgnp-abstract} holds, for suitable choices of $X$, $Y$, and $\TT$. In the whole subsection, let $X=\BB_0(\B_N)$ be the little Bloch space of $\B_N$, and let $Y$ be the space $L^0(\S_N)$ of all (equivalent classes of) functions $m_N$-measurable on $\S_N$ (with identification of two functions that coincide off a set of $m$-measure $0$), endowed with the translation-invariant metric
\[
d(g,h)=\int_{\S_N}\min{\left(1,|g-h|\right)}dm_N.
\]
We recall that a sequence $(g_n)_n$ in $L^0(\S_N)$ converges in measure to $g\in L^0(\S_N)$ if and only if every subsequence of $(g_n)_n$ has a subsequence that converges $m_N$-a.e. to $g$ on $\S_N$. In particular, this implies that the set $\PP$ of all polynomials is dense in $L^0(\S_N)$.

Let us fix a sequence $(r_n)_n$, $r_n\in (0,1)$, converging to $1$. For $w \in \B_N$, let us define $\TT_{w}$ as the sequence $(T_n^{w})_n$ where
\[
T_n^{w}:=\left\{\begin{array}{cll}\BB_0(\B_N) & \to & L^0(\S_N)\\
f & \mapsto & f(r_n(\cdot-w)+w).
\end{array}\right.
\]
Note that if $L\subset \B_N$ is compact then, by uniform continuity, the set
\[
\{f\in \BB_0(\B_N):\,\sup_{w\in L}d(T_n^w(f)(z),f_0)<\varepsilon\}
\]
is open, for every $f_0\in \BB_0(\B_N)$, every $\varepsilon>0$ and every $n\in \N$. In particular, condition \eqref{item1} of Theorem \ref{thm-bgnp-abstract} is always satisfied in this setting. The same argument shows that for every $P\in \PP$, the sequence $(T_n^w(P))_n$ converges to $P\in L^0(\S_N)$, uniformly with respect to $w\in L$.

In order to apply Theorem \ref{thm-bgnp-abstract}, it thus remains to check condition \eqref{item3}. The latter is a consequence of Lemma \ref{lemma-simul-approx-Bloch}. Indeed, fix 
$\varepsilon>0$ and $\vp \in L^0(\S_N)$. By Lusin's theorem, every measurable function on $\S_N$ can be approximated $m_N$-a.e. by continuous functions, so we can assume that $\vp \in C(\S_N)$. We can now apply Lemma \ref{lemma-simul-approx-Bloch}, with $g=0$, to get a compact set $E\subset \S_N$, with $m_N(E)\geq 1-\varepsilon/2$, and a polynomial $f\in \BB_0(\B_N)$ such that
\[
\|f\|_{B_0(\B_N) }< \varepsilon\quad \text{and}\quad \|f-\vp\|_{\infty,E}<\varepsilon/2.
\]
Since $f$ is uniformly continuous on $\overline{\B_N}$, the second inequality implies the existence of some $n\in \N$ such that
\[
\sup_{w\in L}\|T_n^w(f)-\vp\|_{\infty,E}<\varepsilon/2
\]
To see that condition \eqref{item3} of Theorem \ref{thm-bgnp-abstract} holds, it suffices to observe that
\begin{multline*}
\sup_{w\in L}d(T_n^w(f),\vp)=\sup_{w\in L}\int_{\S_N}\min\left(1,|T_n^i(f)-\vp|\right)dm_N\\
\leq \sup_{w\in L}\int_{E}\min\left(1,|T_n^w(f)-\vp|\right)dm_N + m_N(\S_N\setminus E) < \varepsilon.
\end{multline*}

By Theorem \ref{thm-bgnp-abstract}, we deduce that the set $\UU(\TT_{w\in L})$ is residual in $\BB_0(\B_N)$. Now, considering an exhaustion $(L_n)_n$ of $\B_N$ by compact sets, we obtain that the set
\[
\bigcap_{n\in \N}\UU(\TT_{w\in L})
\]
is residual in $\BB_0(\B_N)$. Using again that a sequence $(g_n)_n$ in $L^0(\S_N)$ converges in measure to $g\in L^0(\S_N)$ if and only if every subsequence of $(g_n)_n$ has a subsequence that converges $m$-a.e. to $g$ on $\S_N$, we deduce that the previous implies the assertion (1) of the main theorem, announced in the introduction. More precisely,

\begin{theorem}\label{main-thm-numbered}Let $(r_n)_n$ be a sequence of real numbers in $(0,1)$, converging to $1$. There exists a function $f$ in $\BB_0(\B_N)$ with the following property $(\mathcal{Q}_{\B_N})$: given any $m$-measurable function $\vp$ on $\S_N$, there exists a sequence $(n_k)_k$ of integers such that for any $w\in \B_N$ and $m_N$-a.e. $\zeta \in \S_N$,
	\[
	f(r_{n_k}(\zeta-w)+w)\to \vp(\zeta) \text{ as }k\to \infty.
	\]
	The set of such functions is residual in $\BB_0(\B_n)$.
\end{theorem}

\subsection{Wild Bloch functions in the polydisc}\label{sub-wild-polydisc}

Similarly to the previous subsection, we will now derive from Lemma \ref{lemma-simul-approx-Bloch-polydisc} and Theorem \ref{thm-bgnp-abstract} a way to exhibit a large set of function in $\BB_0(\D^N)$ with universal boundary behaviour. We set $X=\BB_0(\D^N)$ , $Y=L^0(\T^N)$, that is the space of all (equivalent classes of) functions $\sigma_N$-measurable on $\T^N$ (with equivalence relation identifying two functions that coincide off a set of $\sigma_N$-measure $0$), endowed with the translation-invariant metric
\[
d(g,h)=\int_{\T^N}\min{\left(1,|g-h|\right)}d\sigma_N.
\]
As recalled in Subsection \ref{sub-wild-ball}, a sequence $(g_n)_n$ in $L^0(\T^N)$ converges in measure to $g\in L^0(\T^N)$ if and only if every subsequence of $(g_n)_n$ has a subsequence that converges $\sigma_N$-a.e. to $g$ on $\T^N$. Thus the set $\PP$ of all polynomials is dense in $L^0(\T^N)$.

For a sequence $(r_n)_n$, $r_n\in (0,1)$, converging to $1$ and $w \in \B_N$, let $\TT_{w}$ be the sequence $(T_n^{w})_n$ where
\[
T_n^{w}:=\left\{\begin{array}{cll}\BB_0(\D^N) & \to & L^0(\T^N)\\
	f & \mapsto & f(r_n(\cdot-w)+w).
\end{array}\right.
\]
For $L\subset \D$ compact, let us write $L^N=L\times \ldots \times L$ ($N$ times). By uniform continuity, the set
\[
\{f\in \BB_0(\D^N):\,\sup_{w\in L^N}d(T_n^w(f)(z),f_0)<\varepsilon\}
\]
is open, for every $f_0\in \BB_0(\D^N)$, every $\varepsilon>0$ and every $n\in \N$. Thus condition \eqref{item1} of Theorem \ref{thm-bgnp-abstract} is valid in this setting. For the same reason, for every $P\in \PP$, the sequence $(T_n^w(P))_n$ converges to $P\in L^0(\T^N)$, uniformly with respect to $w\in L^N$.

Therefore we may apply Theorem \ref{thm-bgnp-abstract}, whenever we have proven that condition \eqref{item3} holds true as well. This is where Lemma \ref{lemma-simul-approx-Bloch-polydisc} comes into play. First, Lusin's theorem ensures that we may assume that $\vp \in C(\T^N)$.  Then, if $\varepsilon>0$ and $\vp \in C(\T^N)$ are given, Lemma \ref{lemma-simul-approx-Bloch-polydisc} applied with $g=0$, provides us with a compact set $E\subset \T^N$ and a polynomial $f\in \BB_0(\D^N)$ such that $\sigma_N(E)\geq 1-\varepsilon/2$ and
\[
\|f\|_{B_0(\D^N) }< \varepsilon\quad \text{and}\quad \|f-\vp\|_{\infty,E}<\varepsilon/2.
\]
By uniform continuity of $f$ on $\overline{\D^N}$, we get from the second inequality the existence of some $n\in \N$ such that
\[
\sup_{w\in L^N}\|T_n^w(f)-\vp\|_{E,\infty}<\varepsilon/2.
\]
Next, observe that
\begin{multline*}
	\sup_{w\in L^N}d(T_n^w(f),\vp)=\sup_{w\in L^N}\int_{\T^N}\min\left(1,|T_n^w(f)-\vp|\right)d\sigma_N\\
		\leq \sup_{w\in L^N}\int_{E}\min\left(1,|T_n^w(f)-\vp|\right)d\sigma_N + \sigma_N(\T^N\setminus E) < \varepsilon,
		\end{multline*}
which is condition \eqref{item3} of Theorem \ref{thm-bgnp-abstract}.

Finally,  Theorem \ref{thm-bgnp-abstract} tells us that the set $\UU(\TT_{w\in L^N})$ is residual in $\BB_0(\D^N)$. Let $(L_n)_n$ be an exhaustion of $\D$ by compact sets. Then $(L_n^N)_n$ is an exhaustion by $\D^N$ of compact sets and, by the Baire theorem, we conclude that
\[
\bigcap_{n\in \N}\UU(\TT_{w\in L_n^N})
\]
is residual in $\BB_0(\D^N)$. We derive the following theorem.

\begin{theorem}\label{main-thm-numbered-polydisc}Let $(r_n)_n$ be a sequence of real numbers in $(0,1)$, converging to $1$. There exists a function $f$ in $\BB_0(\D^N)$ satisfying the following property $(\QQ_{\D^N})$: given any $\sigma_N$-measurable function $\vp$ on $\T^N$, there exists a sequence $(n_k)_k$ of integers such that for any $w\in \D^N$ and $\sigma_N$-a.e. $\zeta \in \T^N$,
	\[
	f(r_{n_k}(\zeta-w)+w)\to \vp(\zeta) \text{ as }k\to \infty.
	\]
Moreover, the set of such functions is residual in $\BB_0(\D^n)$.
\end{theorem}

Note that Theorem \ref{main-thm-numbered} and Theorem \ref{main-thm-numbered-polydisc} gives the same statement when $N=1$.

\section{Further extensions and variations}
\subsection{Weighted Bloch spaces}
Let  $\omega:(0,1) \to (0,+\infty)$ be a non-decreasing weight (that is $\omega(t)\to 0$ as $t\to 0^+$). The weighted Bloch space $\BB_{\omega}(\D)$ associated with $\omega$ is defined as the space of all functions $f$ holomorphic in $\D$, such that
\[
\|f\|_{\BB_\omega(\D)}:=|f(0)|+\sup_{|z|<1}\frac{1-|z|^2}{\omega(1-|z|)}|f'(z)|<\infty,
\]
and the little weighted Bloch space $\BB_{0,\omega}(\D)$ as the closure of all polynomials in $\BB_{\omega}(\D)$ or, equivalently, as the closed subspace consisting of all $f$ such that
\[
\frac{1-|z|}{\omega(1-|z|)}|f'(z)|\to 0\quad \text{as }|z|\to 1^-.
\]

Further, we can also define the weighted Bloch space of $\D^N$ as the Banach space of all functions $f\in H(\D)$ that satisfy
\[
\|f\|_{\BB_\omega(\D^N)}:=|f(0)|+\sup_{|z|<1}\sum_{k=1}^N\frac{1-|z_k|^2}{\omega(1-|z_k|^2)}\left|\frac{\partial f}{\partial z_k}(z)\right|<\infty,
\]
and the little Bloch space $\BB_{0,\omega}(\D^N)$ as the closure of the set of polynomials in $\BB_{\omega}(\D^N)$.

\medskip

Thus, replacing Theorem \ref{thm-inner-function-ball} by \cite[Theorem 5.2]{AAN1999}, we may obtain a version of Corollary \ref{cor-approx-disc} for $N=1$, with $\BB_{\omega}(\D)$ instead of $\BB(\D)$, and then obtain the same statement as Lemma \ref{lemma-simul-approx-Bloch-polydisc} for $\BB_{0,\omega}(\D^N)$. The same lines as those leading to Theorem \ref{main-thm-numbered-polydisc} in Subsection \ref{sub-wild-polydisc} may finally lead to the following improvement.

\begin{theorem}\label{thm-Bloch-general-disc}Let $\omega$ be a non-decreasing weight such that, for some $\varepsilon>0$, the function $\omega(t)/t^{1-\varepsilon}$ is decreasing. If
	\[
	\int_x^1\frac{\omega(t)^2}{t}dt=+\infty,
	\]
then there exists a residual set of functions in $\BB_{0,\omega}(\D^N)$ satisfying the property $(\QQ_{\D^N})$ (see Theorem \ref{main-thm-numbered-polydisc}).
\end{theorem}

Interestingly, the integral condition in the previous result is sharp. Indeed, Proposition 1.1 in \cite{Doubtsov2014} shows that if $\int_x^1\frac{\omega(t)^2}{t}dt< +\infty$, then every function in $\BB_{\omega}(\D)$ has radial limits almost everywhere. Since the latter result holds true for the unit ball $\B_N$ as well, it is natural to wonder whether Theorem \ref{thm-Bloch-general-disc} also holds for $\BB_0(\B_N)$, $N\geq 2$ (for some weight function $\omega$ defined on $\B_N$).

\subsection{Cluster sets of Bloch functions along curves}
For $\zeta \in \S_N$, let us say that a (continuous) path $\gamma:[0,1) \to \B_N$ \textit{terminates non-tangentially} at $\zeta$ if $\gamma(t)\to \zeta$ as $t\to 1$, and there exists a compact set $L\subset \B_N$ such that for any $t\in [0,1)$, $\gamma(t)$ belongs to the set
\[
\left\{r(\zeta-w)+w:\,w\in L,\, 0\leq r\leq 1\right\}.
\]

In Subsection \ref{sub-wild-ball}, it is shown the existence of a residual set of functions $f$ with the property that, given any $\varepsilon>0$, any $\vp\in C(\S_N)$ and any compact set $L\subset \B_N$, there exists $n\in \N$ and $1-\varepsilon <r<1$, such that for $m_N$-a.e. $\zeta \in \S_N$,
\[
\sup_{w\in L}|f(r(z-w)+w)-\vp(z)|<\varepsilon.
\]

Since for any path $\gamma$ terminating non-tangentially at some $\zeta\in \S_N$, one has $\gamma([0,1))\cap \{r(\zeta-w)+w:\,w\in L\}\neq \emptyset$, for any $r\in [0,1)$ large enough, we immediately deduce the following corollary.

\begin{cor}\label{coro-end}There exists a residual set in $\BB_0(\B_N)$ consisting of functions $f$ that satisfy the following property: there exists a set $E\subset \S_N$, with $m_N(E)=1$, such that for any path $\gamma$ terminating non-tangentially at a point of $E$, the cluster set of $f$ along $\gamma$ is maximal - that is the set $f(\gamma([0,1)))$ is dense in $\C$.
\end{cor}

This result complements the fact that, for every function $f$ in $\BB(\B_N)$, there exists a dense set of points in $\S_N$ at which $f$ admits (finite or infinite) radial limit \cite{BagemihlSeidel1961}. For $N=1$, since every Bloch function is normal, it turns out that every radial limit of a non-constant function in $\BB(\D)$ is also a non-tangential limit \cite{LehtoVirtanen1957}. In contrast, we recall that there exists a residual set in $H(\B_N)$ of functions having a maximal cluster set along any path terminating at a point of $\S_N$, with finite length \cite{CharpentierKosinski2021}. For the case of the disc, the assumption \textit{with finite length} can be dropped (see \cite{BoivinGauthierParamonov2002} or \cite{Charpentier2020}).

A formulation of Corollary \ref{coro-end} for the polydisc could be naturally formulated. For the sake of brevity, we do not do it here and leave it to the interested reader.
\bibliographystyle{amsplain}
\bibliography{refs}
	
\end{document}